\documentclass[a4paper,12pt]{article}
\usepackage{amsmath}
\usepackage{amssymb}
\usepackage{setspace}
\usepackage{fullpage}
\usepackage{bbm}
\usepackage{amsthm}
\usepackage[latin1]{inputenc}
\usepackage{xcolor}
\usepackage[normalem]{ulem}
\usepackage{hyperref}
\usepackage{rotating}
\usepackage{graphicx}
\usepackage{enumerate}

\usepackage[normalem]{ulem} 

\usepackage[
singlelinecheck=false 
]{caption}

\newtheorem{theorem}{Theorem}[section]
\newtheorem{lemma}[theorem]{Lemma}

\newtheorem{corollary}[theorem]{Corollary}

\theoremstyle{definition}

\newtheorem{construction}[theorem]{Construction}
\newtheorem{definition}[theorem]{Definition}
\newtheorem{proposition}[theorem]{Proposition}

\newtheorem{remark}[theorem]{Remark}
\newtheorem*{remark*}{Remark}

\newtheorem{example}[theorem]{Example}
\newtheorem*{example*}{Example}

\usepackage{pdfsync}
\def\PG{\mathrm{PG}}  
  
\def\PGammaL{\mathrm{P}\Gamma\mathrm{L}}

\def\F{\mathbb{F}}

\title{On linear sets of minimum size}

\author{Dibyayoti Jena  \thanks{Supported by the Marsden Fund Council administered by the Royal Society of New Zealand.} \and Geertrui Van de Voorde \thanks{Postdoctoral fellow of the Research Foundation Flanders (FWO -- Vlaanderen) and supported by the Marsden Fund Council administered by the Royal Society of New Zealand.}}

\begin{document}
\maketitle
\begin{abstract} An $\F_q$-linear set of rank $k$, $k\le h$, on a projective line $\PG(1,q^h)$, containing at least one point of weight one, has size at least $q^{k-1}+1$ (see \cite{vdv}). The classical example of such a set is given by a {\em club}. 
In this paper, we construct a broad family of linear sets meeting this lower bound, where we are able to prescribe the weight of the heaviest point to any value between $k/2$ and $k-1$. Our construction extends the known examples of linear sets of size $q^{k-1}+1$ in $\PG(1,q^h)$ constructed for $k=h=4$ \cite{bonoli} and $k=h$ in \cite{polverino}. We determine the weight distribution of the constructed linear sets and describe them as the projection of a subgeometry. For small $k$, we investigate whether all linear sets of size $q^{k-1}+1$ arise from our construction.

Finally, we modify our construction to define rank $k$ linear sets of size $q^{k-1}+q^{k-2}+\ldots+q^{k-l}+1$ in $\PG(l,q^ h)$. This leads to new infinite families of small minimal blocking sets which are not of R\'edei type.
\end{abstract}

\section{Introduction}
\subsection{Background and motivation}
Linear sets in finite projective spaces have attracted a lot of research in recent years. They are used in the construction of interesting sets such as blocking sets \cite{polito}, translation ovoids \cite{lunardon}, KM-arcs \cite{maarten}, and have been shown to be useful in the study of other topics such as semifields \cite{olgamichel} and rank metric codes \cite{zullo,jon,zini}. For more background about linear sets and these applications, we refer to \cite{wij,olga}.

More formally, let $\F_{q^h}$ be the finite field of order $q^h$, where $q$ is a prime power. An $\mathbb{F}_q$-\textit{linear set} $L$ \textit{of rank k} in a projective space $\PG(V),V=\mathbb{F}_{q^h}^r$, is a set of points defined by a $k$-dimensional $\mathbb{F}_q$-subspace $U$ of $V$ in the following way:$$L=L_U=\{\langle u\rangle_{q^h}|u\in U\backslash\{0\}\}.$$

The \textit{weight} of a point $P=\langle u_P\rangle_{q^h}$ in a linear set $L_U$ is the vector space dimension of the $\mathbb{F}_q$-subspace $U_P$ of all vectors of $U$ determining the point $P$. That is, $$U_P=\{0\}\cup\{u\in U|\langle u\rangle_{q^h}=\langle u_P\rangle_{q^h}\}.$$ 

Note that $\mathbb{F}_{q^h}^r$ is an $rh$-dimensional vector space over $\mathbb{F}_q$. This induces a natural map $\phi$ from $\PG(r-1,q^h)$ to $\PG(rh-1,q)$. Under this map $\phi$ (which is also called the {\em field reduction map}), points and lines of $\PG(r-1,q^h)$ get mapped to $(h-1)$ and $(2h-1)$-dimensional subspaces of $\PG(rh-1,q)$ respectively. The images of the points under this map form a {\em Desarguesian} $(h-1)$-spread of $\PG(rh-1,q)$. A linear set of rank $k$ is then the set of points of $\PG(r-1,q^h)$ whose images under $\phi$ intersect some fixed $(k-1)$-dimensional subspace $\pi$ in $\PG(rh-1,q)$. In this point of view, the weight of a point $P$ in the linear set is one more than the (projective) dimension of $\phi(P)\cap\pi$. 

Another equivalent way of viewing a linear set of rank $k$ involves the projection of a subgeometry onto a projective space as proven in \cite{lunardon}. Let $\Sigma^*=\PG(k-1,q^h)$, and $\Sigma=\PG(k-1,q)$ be a canonical subgeometry. Let $\Omega=\PG(r-1,q^h)$, and let $\Pi$ be a $(k-r-1)$-dimensional subspace of $\Sigma^*$ disjoint from $\Sigma$ and $\Omega$. The projection map $p_{\Pi,\Omega}:\Sigma\rightarrow\Omega$ is defined by $p_{\Pi,\Omega}(x)=\langle x,\Pi\rangle\cap\Omega$. By \cite[Theorems 1 and 2]{lunardon} every rank $k$ linear set $L$ of $\Omega=\PG(r-1,q^h)$ is either a canonical subgeometry of $\Omega$ or can be viewed as a projection of $\Sigma=\PG(k-1,q)$ from $\Pi$ to $\Omega$, where $\Pi$ is a $(k-r-1)$-dimensional subspace of $\Sigma^*=\PG(k-1,q^h)$ disjoint from $\Sigma$ and $\Omega$. From this point of view the weight of a point $P$ is one more than the dimension of the pre-image of $P$ under the projection map (see also \cite[Proposition 2.7]{jon}).   

\subsection{Linear sets on a line}\label{introline}

A particular case of interest have been the {\em scattered} linear sets on a line: these are linear sets that have the largest possible size when their rank is fixed. The `classical' scattered linear sets of rank $h$ in $\PG(1,q^h)$ arise from the Frobenius map $x\mapsto x^q$ on $\F_{q^h}$. Other examples are known (see e.g. \cite{blokhuis,bence,bence2}) but the classification of scattered linear sets of rank $h$ in $\PG(1,q^h)$ is in general an open problem. For more information about scattered linear sets and their applications, see \cite{lavrauw}. 

On the other side of the size spectrum, we find the smallest possible linear sets when their rank is fixed. For {\em strictly $\F_q$-linear sets} (which are not linear over a superfield of $\F_q$) of rank $h$ in $\PG(1,q^h)$ the lower bound $q^{h-1}+1$ follows from the work of \cite{simeon,bbb} using R\'edei polynomials. Only recently, in \cite[Theorem 3.7]{vdv}, the lower bound $q^{k-1}+1$ was established for linear sets of rank $k\leq h$ {\em containing a point of weight one} in $\PG(1,q^h)$. 

The `classical' example of a rank $k$ linear set of size $q^{k-1}+1$ is called a {\em club}. It contains one point of weight $k-1$ (also called the {\em head}) and the other $q^{k-1}$ points have weight $1$ (see also \cite[Proposition 3.8]{vdv}). Clubs of the same rank are not necessarily equivalent, but they are for $k=h$ (see Subsection \ref{trace}). Other examples of rank $h$ linear sets of size $q^{h-1}+1$ in $\PG(1,q^h)$ were constructed by Lunardon and Polverino \cite{polverino}. These examples contain one point of weight $h-2$, $q^{h-3}$ points of weight $2$ and $q^{h-1}-q^{h-3}$ points of weight $1$. Their example is given by the set 
\begin{align}\{\langle(x_0+x_1\lambda,y_0+y_1\lambda+\dots+y_{h-3}\lambda^{h-3})\rangle\mid (x_0,x_1,y_0,y_1,\dots,y_{h-3})\in(\mathbb{F}_q^h)^{*}\},\label{Lu-Po}\end{align}
where $\lambda$ is a primitive element of $\mathbb{F}_{q^h}$.
For $k=h=4$, this example has size $q^3+1$, has only points of weight $1$ and $2$, and also appears in Bonoli and Polverino \cite{bonoli}. In this paper, we will generalise this example.

\subsection{Linear sets in a plane and linear blocking sets}

Linear sets in $\PG(2,q^h)$ are of a particular interest when their rank is $h+1$: in this case, they form {\em small minimal blocking sets}. A \textit{blocking set in $\PG(2,q_0)$} is a set of points that meets every line of $\PG(2,q_0)$. A blocking set is called \textit{minimal} if it does not contain a proper subset which in itself is a blocking set and it is called \textit{small} if it has less than $\frac{3(q_0+1)}{2}$ points. Finally, a blocking set $B$ in $\PG(2,q_0)$ is said to be of \textit{R\'edei type} is there is a line $\ell$ such that there are precisely $q_0$ points of $B$ that do not on the line $\ell$.
Now recall that a line in $\PG(2,q^h)$ is mapped to a $(2h-1)$-dimensional subspace under the field reduction map $\phi:\PG(2,q^h)\rightarrow\PG(3h-1,q)$ and an $\F_q$-linear set of rank $h+1$ can be determined by an $h$-dimensional subspace of $\PG(3h-1,q)$. Since in $\PG(3h-1,q)$, a $(2h-1)$-dimensional and an $h$-dimensional subspace always intersect, the constructed linear set is clearly a blocking set. Since its size is at most $(q^{h+1}-1)/(q-1)$, it is small. Furthermore, it is not hard to see that every point of the blocking set lies on at least one tangent line to the set, making the blocking set minimal (see e.g. \cite[Lemma 1]{polito}).

By \cite[Theorem 4.1]{vdv} if there is a $(q+1)$-secant, i.e. a line intersecting the linear set in exactly $q+1$ points,  then the minimum size of a rank $k\leq h+1$ linear set in $\PG(2,q^h)$ is $q^{k-1}+q^{k-2}+1$. It is worth noting that $\F_q$-linear blocking sets in $\PG(2,q^h)$ that are not $\F_{q^i}$-linear for some $i>1$ always contain a $(q+1)$-secant by \cite{sziklai}.

Examples of linear blocking sets (i.e. $k=h+1$) of size $q^h+q^{h-1}+1$ in $\PG(2,q^h)$ of both {\em R\'edei and non-R\'edei type} were constructed by Lunardon and Polverino in \cite{polverino}. This extended the work done in \cite{polito}, where the first examples of non-R\'edei type linear blocking sets were constructed, disproving the then commonly held belief that all small minimal blocking sets must be of R\'edei type. 

The weights of the points in the examples of \cite{polverino} are contained in $\{1,2,h-3,h-2,h-1\}$. In our case, we will be able to construct a linear set of rank $k$ of size $q^{k-1}+q^{k-2}+1$ where we can specify the weight of the heaviest points to be any value between $k/3$ and $k-2$.
These examples fall in a more general class of rank $k$ linear sets of size $q^{k-1}+q^{k-1}+\ldots+q^{k-l}+1$ in $\PG(l,q^h)$ which will be presented in Theorem \ref{main}.

\section{Rank $k$ linear sets of size $q^{k-1}+1$ on a line}

\subsection{Construction}

Recall from \cite{lunardon} that we can construct rank $k$ $(1\le k-1\le h)$ linear sets on a line $\Omega=\PG(1,q^h)$ as a projection of a canonical subgeometry $\Sigma=\PG(k-1,q)$ of $\Sigma^*=\PG(k-1,q^h)$ from a suitable $(k-3)$-dimensional subspace $\Pi$, disjoint from $\Sigma$ and $\Omega$, as the axis of the projection.

In the following construction we will use the standard homogeneous coordinates for $\Sigma^*$ and $\Sigma$, i.e. the points in $\Sigma^*$ and $\Sigma$ are of the form $\langle(a_1,a_2,\dots,a_k)\rangle$ and $\langle(b_1,b_2,\dots,b_k)\rangle$ respectively, where $a_i\in \mathbb{F}_{q^h}$ and $b_i\in \mathbb{F}_q$, for $i\in\{1,2,\dots,h\}$. We will use the notation $\textbf{e}_\textbf{j}$ for the vector $(0,\dots,0,1,0,\dots,0)$, where 1 is in the $j_{\text{th}}$ position.

\begin{construction}\label{cons1}
Partition the set $\{\textbf{e}_\textbf{1},\dots,\textbf{e}_\textbf{k}\}$ into $2$ parts $A_1$ and $A_2$ of size $t_1$ and $t_2$ with $t_1,t_2\ge1$. Without loss of generality, let 
\begin{align*}
P_1&=\{\textbf{e}_\textbf{1},\dots,\textbf{e}_{\textbf{t}_\textbf{1}}\}=\{\textbf{e}_{\textbf{1}\textbf{,}\textbf{1}},\dots,\textbf{e}_{\textbf{1}\textbf{,}\textbf{t}_\textbf{1}}\}\\
P_2&=\{\textbf{e}_{\textbf{t}_\textbf{1}\textbf{+}\textbf{1}},\dots,\textbf{e}_{\textbf{t}_\textbf{1}\textbf{+}\textbf{t}_\textbf{2}}\}=\{\textbf{e}_{\textbf{2}\textbf{,}\textbf{1}},\dots,\textbf{e}_{\textbf{2}\textbf{,}\textbf{t}_\textbf{2}}\}
\end{align*}
Note that we consider the above sets as ordered sets. 
Consider $\alpha\in \mathbb{F}_{q^h}\backslash\mathbb{F}_q$ generating a degree $s$-extension of $\F_q$ (i.e. $[\F_q(\alpha):\F_q]=s$), with $k-1\leq s$, $s\mid h$. 

With each $A_i, i=1,2,$ we associate the subspace of $\Sigma^*$  defined as
 $\pi_i=\emptyset$ when $t_i=1$ and,
\begin{align*}
\pi_i&=\langle\langle \textbf{e}_{\textbf{i}\textbf{,}\textbf{1}}-\alpha \textbf{e}_{\textbf{i}\textbf{,}\textbf{2}}\rangle,\langle\textbf{e}_{\textbf{i}\textbf{,}\textbf{2}}-\alpha \textbf{e}_{\textbf{i}\textbf{,}\textbf{3}}\rangle,\dots,\langle\textbf{e}_{\textbf{i},\textbf{t}_{\textbf{i}\textbf{-}\textbf{1}}}-\alpha \textbf{e}_{\textbf{i}\textbf{,}\textbf{t}_{\textbf{i}}}\rangle\rangle\\
&=\langle\langle \textbf{e}_{\textbf{i}\textbf{,}\textbf{1}}-\alpha \textbf{e}_{\textbf{i}\textbf{,}\textbf{2}}\rangle,\langle\textbf{e}_{\textbf{i}\textbf{,}\textbf{1}}-\alpha^2\textbf{e}_{\textbf{i}\textbf{,}\textbf{3}}\rangle,\dots,\langle\textbf{e}_{\textbf{i}\textbf{,}\textbf{1}}-\alpha^{t_i-1}\textbf{e}_{\textbf{i}\textbf{,}\textbf{t}_\textbf{i}}\rangle\rangle,  
\end{align*} 
when $t_i\geq 2$. The angle brackets around a vector denote the projective point determined by that vector, whereas the outer angle brackets denote the subspace spanned by these points.
Note that $\pi_1\cap\pi_2=\emptyset.$
Let $\Omega$ be the subspace $$\Omega:=\langle \langle\textbf{e}_{\textbf{1}\textbf{,}\textbf{t}_\textbf{1}}\rangle,\langle\textbf{e}_{\textbf{2}\textbf{,}\textbf{t}_\textbf{2}}\rangle\rangle$$
and $\Pi$ be the $(k-3)$-dimensional subspace $$\Pi:=\langle\pi_1,\pi_{2}\rangle.$$	
\end{construction}

\begin{lemma}\label{lem1}
The projection of $\Sigma$ from $\Pi$ onto $\Omega$	in Construction \ref{cons1} is an $\F_q$-linear set of rank $k$ in $\Omega=\PG(1,q^h)$.
\end{lemma}

\begin{proof}  We only need to check that $\Pi$ is disjoint from $\Sigma$ and $\Omega$. 
	
Note that for a point $\langle(\lambda_{1,1},\dots,\lambda_{1,t_1},\lambda_{2,1},\ldots,\lambda_{{2},t_{2}})\rangle\in\Pi$, we have 
\begin{equation}
\lambda_{i,1}=-\lambda_{i,2}/\alpha-\lambda_{i,3}/\alpha^2-\dots-\lambda_{i,t_i}/\alpha^{t_i-1},
\label{special}
\end{equation}whenever $t_i\ge2$. If $t_i=1$, then $\pi_i=\emptyset$, forcing $\lambda_{i,t_i}=\lambda_{i,1}=0$.
	
Now assume that there exists a point $$P=\langle(\lambda_{1,1},\dots,\lambda_{1,t_1},\lambda_{2,1},\ldots,\lambda_{{2},t_{2}})\rangle\in\Pi\cap\Sigma.$$ As $P\in\Sigma$ we can also assume that $\lambda_{1,1},\dots,\lambda_{1,t_1},\lambda_{2,1},\ldots,\lambda_{{2},t_{2}}\in \mathbb{F}_q$, not all of which are equal to zero. Then $\lambda_{1,1},\dots,\lambda_{1,t_1}$ or $\lambda_{2,1},\dots,\lambda_{{2},t_{2}}\in \mathbb{F}_q$ are not all zero. Without loss of generality we may assume that the former case holds. Then $t_1\ge2$ and from \eqref{special} $\lambda_{1,1}=-\lambda_{1,2}/\alpha-\lambda_{1,3}/\alpha^2-\dots-\lambda_{1,t_1}/\alpha^{t_1-1},$ contradicting the choice of $\alpha$ since $t_1-1<s$. Hence $\Sigma$ and $\Pi$ are disjoint.
	
Now assume that there exists a point $$P=\langle(0,\dots,0,\lambda_{1,t_1},0,\dots,0,\lambda_{2,t_2})\rangle\in \Omega\cap\Pi,$$ then without loss of generality we may assume that $\lambda_{1,t_1}\ne0$. This implies that $t_1\ge2$ and as $P\in\Pi$, \eqref{special} holds, implying $0=-\lambda_{1,t_1}/\alpha^{t_1-1}$, a contradiction. Thus $\Pi$ is disjoint from $\Omega$.
\end{proof}

In the following lemma, we will use the greatest common divisor of two polynomials with coefficients in $\F_q$ which is only defined up to scalar multiple; in order to avoid this ambiguity, we will take the $\gcd$ to be {\em monic}.

\begin{lemma}\label{lemnew} Let $\alpha\in \F_{q^h}\setminus\F_q$ be an element generating a degree $s$-extension of $\F_q$ (i.e. $[\F_q(\alpha):\F_q]=s$). Consider the polynomials $f_1$, $f_2$, $g_1$, $g_2$ $\in\mathbb{F}_q[X]$ with $f_1,g_1$ monic, $\deg(f_1)\leq t_1-1,\deg(g_1)\leq t_1-1$, $\deg(f_2)\leq t_2-1,\deg(g_2)\leq t_2-1$, $\gcd(f_1,f_2)=\gcd(g_1,g_2)=1$ and 
$$t_1+t_2\leq s+1.$$

Suppose that $f_1(\alpha)g_2(\alpha)=f_2(\alpha)g_1(\alpha)\in \F_{q}(\alpha)$. Then $f_1=g_1$ and $f_2=g_2$.
\end{lemma}

\begin{proof} Note that $1< s\leq h$ and $s\mid h$, and $\F_q(\alpha)\cong \F_q[X]/(z(X))$ where $z$ is a polynomial of degree $s$, so every element of $\F_q(\alpha)$ has a unique representation as a polynomial in $\alpha$ of degree at most $s-1$. We have that $f_1(\alpha)g_2(\alpha)=f_2(\alpha)g_1(\alpha)$ in $\F_q(\alpha)$. So $$f_1(X)g_2(X)=f_2(X)g_1(X)+t(X)z(X)$$ for some polynomial $t(X)$. But $\deg(z)=s$ and $\deg(f_1g_2-f_2g_1)\leq s-1$ so $t(X)=0$, and $f_1(X)g_2(X)=f_2(X)g_1(X)$. Recall that $\gcd(f_1,f_2)=\gcd(g_1,g_2)=1$ and $f_1,g_1$ are monic. So if $f_2=0$, then we immediately have that $f_1=1$, $g_2=0$, $g_1=1$. The same holds true if $g_2=0$. So we may assume that $f_2\neq 0,g_2\neq 0$ and without loss of generality, assume that $\deg(f_1)\leq\deg(g_1)$. 

Hence we have $f_1g_2=f_2g_1$ with $f_2\neq 0,g_2\neq 0$, $\gcd(f_1,f_2)=1$ and $\deg(f_1)\leq\deg(g_1)$, so we find that $f_1\mid g_1$. This implies that  $g_1=h f_1$, for some polynomial $h$. Then $f_1g_2=h f_2f_1$, and hence, $g_2=h f_2$. Since $h\mid g_1$ and $h\mid g_2$, $h\mid \gcd(g_1,g_2)=1$, so $h\in \F_q$. Since $f_1$ and $g_1$ are monic, $h =1$ which proves our claim.
\end{proof}

\begin{definition}
An $s$-tuple of polynomials $(f_1,\dots,f_s)$ is said to be in \textit{reduced form} if the first non-zero polynomial is monic and gcd$(f_1,\dots,f_s)=1$. We will also say that a projective point $P=\langle(f_1(\alpha),\dots,f_s(\alpha))\rangle$ is in \textit{reduced form}  if $(f_1,\dots,f_s)$ is in  reduced form.
\end{definition}

\begin{lemma}\label{lem2} The linear set $L$ constructed in Lemma \ref{lem1} is given by the points with coordinates
\[\langle(\mu_{1,t_1}+\alpha \mu_{1,t_{1}-1}+\dots+\alpha^{t_{1}-1}\mu_{1,{1}},\mu_{2,{t_2}}+\alpha \mu_{2,{t_{2}-1}}+\dots+\alpha^{t_{2}-1}\mu_{2,{1}})\rangle\]
where $\mu_{i,j}$ are arbitrary elements of $\F_q$, not all zero, $1\le t_1, t_2$, $t_1+t_2=k$, $\alpha\in \mathbb{F}_{q^h}\backslash\mathbb{F}_q$ is a an element generating a degree $s$-extension of $\F_q$ (i.e. $[\F_q(\alpha):\F_q]=s$), with $k-1\leq s$, $s\mid h$.

Moreover, if then there is a bijection between the points of the linear set $L$ and the set \[S=\{(f_1,f_2)\mid \deg(f_1)\leq t_1-1,\deg(f_2)\leq t_2-1\text{ and } (f_1,f_2) \text{ is in reduced form,  $1\le t_1\leq t_2$ }\},\] where $f_1$ and $f_2$ are polynomials over $\mathbb{F}_q$.
\end{lemma}

\begin{proof}  It is clear that the line $\Omega$ and the subspace $\Pi$ in Lemma \ref{lem1} are constructed such that a point $\langle(\mu_1,\dots,\mu_k)\rangle=\langle(\mu_{1,1},\dots,\mu_{2,{t_2}})\rangle$ of $\Sigma$ is projected to the point $$\langle f_1(\alpha)\textbf{e}_{\textbf{1}\textbf{,}{\textbf{t}_\textbf{1}}}+f_2(\alpha)\textbf{e}_{\textbf{2}\textbf{,}{\textbf{t}_\textbf{2}}}\rangle$$ in $\Omega$, where
$$f_i(\alpha)=\mu_{i,{t_i}}+\alpha \mu_{i,{t_{i}-1}}+\dots+\alpha^{t_{i}-1}\mu_{i,{1}}$$for each $i\in\{1,2\}$.

This implies that the linear set is indeed given by the points with coordinates \[\langle (\mu_{1,{t_1}}+\alpha \mu_{1,{t_{1}-1}}+\dots+\alpha^{t_{1}-1}\mu_{1,{1}},\mu_{2,{t_2}}+\alpha \mu_{2,{t_{2}-1}}+\dots+\alpha^{t_{2}-1}\mu_{2,{1}})\rangle.\]

Without loss of generality, we can rearrange the coordinates for $L$ such that $t_1\leq t_2$.

 By dividing out a possible non-trivial gcd of the polynomials $f_1,f_2$, we see that every point in the above set can be put in reduced form $(f_1,f_2)$ for some $f_1$ with $deg(f_1)\leq t_1-1$ and $f_2$ with $\deg(f_2)\leq t_2-1$, hence, every point of $L$ gives rise to an element in $S$. Vice versa, if the reduced forms $(f_1,f_2)$ and $(g_1,g_2)$ of $S$ represent the same point in the constructed linear set,
then we have that $f_1 (\alpha)g_2 (\alpha)=f_2 (\alpha)g_1(\alpha)$. By Lemma \ref{lemnew}, we find that $f_1=g_1$ and $f_2=g_2$. Hence, we obtain a bijection between the points of $L$ and the reduced forms contained in $S$.
\end{proof}

In the next lemma, we will count the number of couples of polynomials in reduced form. It was pointed out to us that the problem of counting the number of coprime polynomials was solved in \cite[Proposition 3]{corteel} (see also \cite{benjamin}). Since the proof below is elementary and has a different flavour than the cited ones, we include it here.

\begin{lemma}\label{newlem1}
The number of elements in the set \[S=\{(f_1,f_2)\mid \deg(f_1)\leq t_1-1,\deg(f_2)\leq t_2-1,1\le t_1\leq t_2\text{ and } (f_1,f_2) \text{ is in reduced form}\}\] is $q^{t_1+t_2-1}+1$, where $f_1$ and $f_2$ are polynomials over $\mathbb{F}_q$.
\end{lemma}

\begin{proof}Note that if $f_1=0,$ then for $(f_1,f_2)$ to be in reduced form $f_2$ has to be equal to $1$. So in order to determine the cardinality of $S$ we just need to count the number of couples for which $f_1\ne0$ and then add one.
	
Let $P_{n,\le m}$ be the set of couples $(f_1,f_2)$ with $f_1$ monic of degree $n$ and $\deg(f_2)\leq m$. It is clear that $P_{n,\leq m}$ has size $q^{m+n+1}$. 
Let $R_{n,\le m}$ be the set of couples $(f_1,f_2)$ in $P_{n,\le m}$ in reduced form.

We see that $$S=\cup_{i=0}^{t_1-1} R_{i,\le t_2-1}\cup\{(0,1)\}.$$
Now for a couple $(f_1,f_2)$ in $P_{n,\leq m}$ with gcd$(f_1,f_2)=g$, where $g$ has degree $d$, we see that $(f_1,f_2)$ has reduced form $(f_1',f_2')$ with $\deg(f_1')=n-d$, and hence, determines an element in $R_{n-d, \le m-d}$. 
Vice versa, for a couple $(f_1',f_2')$ in $R_{n-d,\le m-d}$ and a monic polynomial $g'$ of degree $d$, we see that $(f_1'g',f_2'g')$ is in $P_{n,\le m}$. Since there are $q^d$ monic polynomials of degree $d$, every couple $(f_1',f_2')$ in $R_{n-d,\le m-d}$ gives rise to $q^d$ different elements in $P_{n,\leq m}$.

First note that $R_{0,\leq n_0}$ has size $q^{n_0+1}$.
We will show by induction that the number of couples in $R_{m,\le n_m}$ $(1\leq m\le n_m\le t_1-1)$ is:
\begin{equation}
q^{m+n_m+1}-q^{m+n_m}.\label{eq1}
\end{equation}

Let $m=1$. There are $q^{n_1+2}$ couples $(f_1,f_2)$ in $P_{1,\le n_1}$. Of those, we need to exclude the couples that do not belong to $R_{1,\le n_1}$. 
	
As explained above, the excluded couples all arise from the $q^{n_1}$ couples in $R_{0,\le n_1-1}$ which each give rise to $q$ different elements in $P_{1,\le n_1}\setminus R_{1,\le n_1}$. Hence, we indeed find $q^{n_1+2}-q^{n_1+1}$ couples in $R_{1,\le n_1}$.
	
Now assume that the induction hypothesis holds for all $R_{m,\le n_m}$ with $m\leq m_0-1$ for some $1\leq m_0-1\leq t_1-2$. 
	
We will show that the number of elements in $R_{m_0,\le n_{m_0}}$ is $q^{m_0+n_{m_0}+1}-q^{m_0+n_{m_0}}$. 
We know that the number of elements in $P_{m_0,\le n_{m_0}}$ is $q^{m_0+n_{m_0}+1}$. As before, we need to subtract from these all couples arising from elements of $R_{m_0-1,\le n_{m_0}-1}$, which each gives rise to $q$ couples in $P_{m_0,\le n_{m_0}}\setminus R_{m_0,\le n_{m_0}}$, all couples from elements of $R_{m_0-2,\le n_{m_0}-2}$, each giving rise to $q^2$ couples in $P_{m_0,\le n_{m_0}}\setminus R_{m_0,\le n_{m_0}}$, and in general, all couples from $R_{m_0-d,\le n_{m_0}-d}$, each giving rise to $q^d$ couples in $P_{m_0,\le n_{m_0}}\setminus R_{m_0,\le n_{m_0}}$. Note that all excluded couples are distinct as every element of  $P_{m_0,\leq n_{m_0}}$ has a unique reduced form. Hence,

$$|R_{m_0,\le n_{m_0}}|=$$ $$q^{m_0+n_{m_0}+1}-q|R_{m_0-1,\le n_{m_0}-1}|-q^2|R_{m_0-2,\le n_{m_0}-2}|-\cdots q^{m_0-1}|R_{1,\le n_{m_0}-m_0+1}|-q^{m_0}|R_{0,\le n_{m_0}-m_0}|.$$
	
Using the induction hypothesis, we see that
\[|R_{m_0,\le n_{m_0}}|=q^{m_0+n_{m_0}+1}-q(q^{m_0+n_{m_0}-1}-q^{m_0+n_{m_0}-2})-q^2(q^{m_0+n_{m_0}-3}-q^{m_0+n_{m_0}-4})-\cdots\]\[ q^{m_0-1}(q^{n_{m_0}-m_0+3}-q^{n_{m_0}-m_0+2})-q^{m_0}q^{n_{m_0}-m_0+1}=q^{m_0+n_{m_0}+1}-q^{m_0+n_{m_0}},\]
which we needed to show.
	
Now recall that $S=\cup_{i=0}^{t_1-1} R_{i,\le t_2-1}\cup\{(0,1)\}$, $R_{i,\le t_2-1}\cap R_{j,\le t_2-1}=\emptyset$ when $i\ne j$ and $(0,1)\notin R_{i,t_2-1}$ for any $i$.
	
Using $n_0=n_1=\dots=n_{t_1-1}=t_2-1$ in the expression found for $R_{m,\le n_m}$ and summing for $m$ between $0$ and $t_1-1$, we get that there are $(q^{t_1+t_2-1}-q^{t_1+t_2-2})+(q^{t_1+t_2-2}-q^{t_1+t_2-3})+\dots+(q^{t_2+1}-q^{t_2})+q^{t_2}+1=q^{k-1}+1$ couples in $S$. 	
\end{proof}

\begin{theorem}\label{thm1}Let $\alpha\in \F_{q^h}\setminus\F_q$ be an element generating a degree $s$-extension of $\F_q$ (i.e. $[\F_q(\alpha):\F_q]=s$) and

\[ L=\langle(\mu_{1,{t_1}}+\alpha \mu_{1,{t_{1}-1}}+\dots+\alpha^{t_{1}-1}\mu_{1,{1}},\mu_{2,{t_2}}+\alpha \mu_{2,{t_{2}-1}}+\dots+\alpha^{t_{2}-1}\mu_{2,{1}})\rangle\]
where $\mu_{i_j}$ are arbitrary elements of $\F_q$, not all zero, $1\leq t_1,t_2$, $t_1+t_2=k\leq s+1$. Then $L$ is an $\F_q$-linear set of $\PG(1,q^h)$ of rank $k$ with $q^{k-1}+1$ points. Let $t_1\leq t_2$, then there is one point of weight $t_2$ different from $q^{t_2-t_1+1}$ points of weight $t_1$ and $q^{k-2i+1}-q^{k-2i-1}$ points of weight $i$ for $i\in \{1,2,\dots,t_1-1\}$.
\end{theorem}

\begin{proof} \text{Lemma }\ref{lem2} establishes a bijection between the points of the constructed linear set and the set \[S=\{(f_1,f_2)\mid \deg(f_1)\leq t_1-1,\deg(f_2)\leq t_2-1,t_1\leq t_2 \text{ and } (f_1,f_2) \text{ is in reduced form}\}.\] Lemma \ref{newlem1} shows that the number of elements in $S$ is equal to $q^{k-1}+1$. Hence the size of $L$ is $q^{k-1}+1$. We now turn our attention to the weight distribution of the points in $L$.
The point $\langle(0,1)\rangle$ clearly has weight $t_2$. To find the weight of the other points, we need to find out the number of times a point with reduced form in $S$ is determined by a couple in $T$, where 
$T=\cup_{i=0}^{t_1-1}P_{i,\leq t_2-1}$,  where $S, P_{i,\le j}$ and $R_{i,\le j}$ are as in Lemma \ref{newlem1}.
 
Let $(g_1,g_2)$ be the reduced form of a point $P$. We know that the weight of $P$ in $L$ is one if there is no couple $(h_1,h_2)$ in $T$ with $(h_1,h_2)\neq (g_1,g_2)$ and $P=\langle(h_1(\alpha),h_2(\alpha))\rangle$. It is easy to see that the weight of a point $P$ is $w$ when there are $q^{w-1}$ couples in $T$ whose reduced form is $(g_1,g_2)$.

This implies that all the points determined by couples in $R_{t_1-1,\leq t_2-1}$ are of weight $1$. The points with reduced form $(f_1,f_2)$ in $R_{t_1-2,\leq t_2-1}$ have weight $1$ if $\deg(f_2)=t_2-1$ and weight $2$ otherwise.

Likewise, we see that the points with reduced form $(f_1,f_2)$ in $R_{t_1-3,\leq t_2-1}$ have weight $1$ if the $\deg(f_2)=t_2-1$, have weight $2$ if $\deg(f_2)=t_2-2$ and weight $3$ if $\deg(f_2)=t_2-3$.

Table \ref{table1} summarises the condition on the degree of $f_2$ for obtaining different weight points when $f_1$ is non-zero. Here, we represent the degree of $f_i$ by $d_i$, for $i=1,2$.

\begin{table}
	\resizebox{\textwidth}{!}{
		\begin{tabular}{|c||c|c|ccc|c|c|}
			\hline 
			
			$d_1$& weight 1 & weight 2   &&$\dots$ & & weight $t_1-1$  & weight $t_1$  \\ 
			\hline
			\hline 
			$t_1-1$& $d_2\le t_2-1$  &0  &&$\dots$& &0  &0  \\ 

			\hline 
			$t_1-2$& $d_2= t_2-1$ & $d_2\le t_2-2$ &&$\dots$&   &0  &0  \\ 
	
			\hline 
			$t_1-3$& $d_2=t_2-1$ & $d_2=t_2-2$ &&$\dots$&   &  0&  0\\ 
	
			\hline 
	
			$\vdots$& $\vdots$ & $\vdots$ && $\ddots$ &&$\vdots$ & $\vdots$  \\ 
		
			\hline 
			1& $d_2=t_2-1$ &$d_2=t_2-2$  &   &$\dots$& & $d_2\le t_2-(t_1-1)$ & 0 \\ 
			\hline 
			0& $d_2=t_2-1$ & $d_2=t_2-2$ & &$\dots$& & $d_2=t_2-(t_1-1)$  & $d_2\le t_2-t_1$ \\ 
		
			\hline 
		\end{tabular}
	}
	\smallskip
	\caption{Conditions on the degree of $f_2$, $d_2$, for obtaining different weight points when $f_1$ is non-zero of degree $d_1$. }\label{table1}
\end{table}

By looking at the first column and using Equation (\ref{eq1}), we see that the total number of points of weight 1 is $(q^{t_1+t_2-1}-q^{t_1+t_2-2})+(q^{t_1+t_2-2}-2q^{t_1+t_2-3}+q^{t_1+t_2-4})+\dots+(q^{t_2+1}-2q^{t_2}+q^{t_2-1})+(q^{t_2}-q^{t_2-1})=q^{k-1}-q^{k-3},$ where we used that $k=t_1+t_2$. 

Similarly, the total number of points of weight 2 is $(q^{t_1+t_2-3}-q^{t_1+t_2-4})+(q^{t_1+t_2-4}-2q^{t_1+t_2-5}+q^{t_1+t_2-6})+\dots+(q^{t_2}-2q^{t_2-1}+q^{t_2-2})+(q^{t_2-1}-q^{t_2-2})=q^{k-3}-q^{k-5}.$

The total number of points of weight $t_1-1$ is
$q^{k-2t_1+3}-q^{k-2t_1+2}+q^{k-2t_1+2}-q^{k-2t_1+1}=q^{k-2t_1+3}-q^{k-2t_1+1}$ and the total number of points of weight $t_1$ is
$q^{k-2t_1+1}$.

In conclusion, we see that there are $q^{k-2i+1}-q^{k-2i-1}$ points of weight $i$ for $i\in \{1,2,\dots,t_1-1\}$, $q^{t_2-t_1+1}$ points of weight $t_1$ and 1 point, different from the previously counted points, of weight $t_2$.
\end{proof}

\begin{corollary} Let $2\le k\leq h$ and let $j$ be an integer between $k/2$ and $k-1$, then there is a linear set of rank $k$ in $\PG(1,q^h)$ of size $q^{k-1}+1$ whose heaviest point has weight $j$.
\end{corollary}

\subsection{Linear sets of size $q^{k-1}+1$ in $\PG(1,q^h)$ for $k\in \{2,3,4,5,h\}$}

We give an overview of some special cases below. In a few cases, we are able to show that all linear sets of minimum size arise from our construction. 

\subsubsection{$k=h$}\label{trace}
For $h=k$ and partition sizes $1$ and $h-1$ the linear set constructed in Lemma \ref{lem1} is a club of rank $h$ in $\PG(1,q^h)$. Recall that two linear sets, say $L_1$ and $L_2$ in $\PG(l,q^h)$ are called {\em equivalent} if there is an element $\phi$ of $\PGammaL(l+1,q^h)$ such that $\phi(L_1)=L_2$.
Every $\F_q$-linear set of rank $h$ in $\PG(1,q^h)$ containing a point of weight $h$ is equivalent to the set $\{\langle(x,\mathrm{Tr}(x))\rangle\mid x\in \F_{q^n}\}$ (see e.g. \cite[Theorem 2.3]{maarten}). Moreover, it has been shown in \cite[Theorem 3.7]{classes} that this linear set is {\em simple}: if some vector space $U_f=\{(x,f(x))\mid x\in \F_{q^h}\}$ defines the linear set as the vector space $V=\{(x,\mathrm{Tr}(x))\mid x\in \F_{q^n}\}$, then $U_f=\lambda V$ for some $\lambda\in \F_q^*$. 

For $h=k$ and partition sizes $2$ and $h-2$, the linear set constructed in Lemma \ref{lem1} is clearly equivalent to the R\'edei line of the linear set constructed by Lunardon and Polverino \cite[Theorem 3]{polverino} (See \eqref{Lu-Po} in Subsection \ref{introline}). As far as the authors are aware, for all choices of $k\neq h$ and partitions not containing a set of size $1$ or $2$, the linear set constructed in Theorem \ref{thm1} did not appear in the literature.

\subsubsection{$k=2,3$}\label{k23}

Now consider a linear set in $\PG(1,q^h)$ of size $q^{k-1}+1$ with $k=2$. This is simply a subline which indeed is projectively equivalent to the set $\{\langle(\mu_1,\mu_2)\rangle|\mu_1,\mu_2\in \F_q)\}$, showing that it arises from Theorem \ref{thm1} by putting $t_1=t_2=1$.

Every linear set of rank $3$ of size $q^2+1$ corresponds to the projection of a subplane $\Sigma\cong \PG(2,q)$ contained in $\PG(2,q^h)$ from a point $\Pi$ onto a line $\Omega$, such that the point $\Pi$ lies on an extended line of $\Sigma$. Without loss of generality, $\Sigma$ is the canonical subgeometry with coordinates in $\F_q$, the point $\Pi$ is lying on the extended line with equation $y=0$, and hence, is of the form $\langle(1,0,\alpha)\rangle$ for some $\alpha\in \F_{q^h}\setminus \F_{q}$. When projecting $\Sigma$ onto the line $x=0$, we see that the points in the linear set $L$ have coordinates of the form $\langle(\mu_1,\mu_2+\mu_3\alpha)\rangle$, and hence, can be obtained as in Theorem \ref{thm1} by putting $t_1=1,t_2=2$. Note that if $\F_q(\alpha)=\F_{q^2}$, then $L\cong \PG(1,q^2)$, whereas otherwise, $L\not \cong \PG(1,q^2)$. This shows that it is possible to create two non-equivalent linear sets $L_1$ and $L_2$ from our construction. More explicitely, consider the sets $L_1=\{\langle(\mu_1,\mu_2+\mu_3\alpha)\rangle\mid \mu_1,\mu_2,\mu_3\in \F_p, (\mu_1,\mu_2,\mu_3)\neq (0,0,0)\}$ and $L_2=\{\langle(\nu_1,\nu_2+\nu_3\beta)\rangle\mid \nu_1,\nu_2,\nu_3\in \F_p, (\nu_1,\nu_2,\nu_3)\neq (0,0,0)\}$ in $\PG(1,p^6)$, $p$ prime, where $[\F_p(\alpha):\F_p]=2$ and $[\F_p(\beta):\F_p]\in \{3,6\}$. Now let $\phi$ be any element of $\PGammaL(2,p^h)\cong \mathrm{PGL}(2,p^h)\rtimes\mathrm{Aut}(\F_p)$. We see that the cross-ratio of every four points in $L_1$ is contained in $\F_{p}(\alpha)=\F_{p^2}$. Recall that $\mathrm{PGL}(2,p^h)$ preserves cross-ratios, and every field automorphism of $\F_{p^h}$ is preserves subfields. Hence, for every four points in the set $\phi(L_1)$, the cross-ratio will be an element of $\F_{p^2}$. But we see that $\langle (1,0)\rangle,\langle(0,1)\rangle,\langle(1,1)\rangle,\langle(1,\beta)\rangle$ are contained in $L_2$, and the cross-ratio of those four points is $\beta\notin \F_{p^2}$. Hence, there cannot exist an element $\phi\in \PGammaL(2,p^h)$ such that $\phi(L_1)=L_2$.
\subsubsection{$k=4$}

For a linear set of rank $k$ of size $q^{k-1}+1$, we have that \begin{align}\sum_{i=1}^{k}x_i=q^{k-1}+1\label{trivial2} \end{align} and
\begin{align}\sum_{i=1}^{k}x_i\frac{q^i-1}{q-1}=\frac{q^k-1}{q-1}\label{trivial},\end{align} where $x_i$ is the number of points of weight $i$. It follows from \eqref{trivial2} and \eqref{trivial}  that if $L$ is a linear set of rank $4$ of size $q^3+1$, then either
\begin{enumerate}
\item[(a)] $L$ contains $1$ point of weight $3$ and the other $q^3$ points have weight one (i.e. $L$ is a club).
\item[(b)] $L$ contains $q+1$ points of weight $2$ and the other $q^3-q$ points have weight one.
\end{enumerate}
We will show in Proposition \ref{prop1} that all linear sets of rank $4$ of size $q^3+1$ of type (b) arise from our construction.

\begin{proposition} \label{prop1}Let $L$ be a linear set of rank $4$  in $\PG(1,q^h)$ that has size $q^3+1$ and contains $q+1$ points of weight $2$. Then $L$ can be obtained from the construction of Theorem \ref{thm1}.
\end{proposition}
\begin{proof}Let $L$ be a linear set of rank $4$ with size $q^3+1$ in $\PG(1,q^h)$, containing $q+1$ points of weight $2$ and $q^3-q$ of weight $1$. Then $L$ can be obtained as the projection of a subgeometry $\Sigma\cong \PG(3,q)$ from a line, say $\Pi$, disjoint from $\Sigma$, onto some line $\Omega$ disjoint from $\Pi$. Each of the $q+1$ points of weight $2$, say $P_1,\ldots,P_{q+1}$ in $L$, has the property that $\langle P_i,\Pi\rangle$ meets the subgeometry $\Sigma$ in a set of $q+1$ points, forming a subline. Let $n_i$ denote the line of $\PG(3,q^h)$ containing these points. Then $n_i$ intersects $\Pi$ in a point, say $Q_i$. Without loss of generality, $\Sigma$ can be taken to be the canonical subgeometry of $\PG(3,q^h)$, determined by the points that have a set of homogeneous coordinates with all entries in $\F_q$. Moreover, we can take $n_1$ to contain the points with coordinates $\langle(1,0,0,0)\rangle$ and $\langle(0,1,0,0)\rangle$ in $\Sigma$, and $n_2$ to contain the points with coordinates $\langle(0,0,1,0)\rangle$ and $\langle(0,0,0,1)\rangle$ in $\Sigma$.
This implies that $Q_1$ is of the form $\langle(1,\alpha,0,0)\rangle$ for some $\alpha\in \F_{q^h}\setminus \F_q$ since $M$ is disjoint from $\Sigma$. Likewise, $Q_2$ is of the form $\langle(0,0,1,\beta)\rangle$ for some $\beta\in \F_{q^h}\setminus \F_q$.
The line $n_3$ of $\PG(3,q^h)$ intersects $\Sigma$ in a subline which is disjoint from $n_1\cap \Sigma$ and $n_2\cap \Sigma$. Hence, $n_3$ intersects the subplane of $\Sigma$ spanned by $\langle (1,0,0,0)\rangle,\langle(0,0,1,0)\rangle,\langle(0,0,0,1)\rangle$ in exactly one point, which then has coordinates $\langle(1,0,a,b)\rangle$ for some $a,b\in \F_q$; similarly, it intersects the subplane of $\Sigma$ through $\langle (0,1,0,0)\rangle,\langle(0,0,1,0)\rangle,\langle(0,0,0,1)\rangle$ in exactly one point, which then has coordinates $\langle(0,1,c,d)\rangle$ for some $c,d\in \F_q$. Since $n_1,n_2,n_3$ are disjoint, it follows that $ad-bc\neq 0$.
The point $Q_3$ lies on the line of $\PG(3,q^h)$ through $\langle(1,0,a,b)\rangle$ and $\langle(0,1,c,d)\rangle$, and also on the line of $\PG(3,q^h)$ through $Q_1=\langle(1,\alpha,0,0)\rangle$ and  $Q_2=\langle(0,0,1,\beta)\rangle$. Expressing that $\langle(1,0,a,b)\rangle, \langle(0,1,c,d)\rangle,\langle(1,\alpha,0,0)\rangle$ and $\langle(0,0,1,\beta)\rangle$ are coplanar forces $\alpha=\frac{a\beta-b}{d-c\beta}$.

Now consider the collineation $\psi$ induced by the matrix $\begin{bmatrix}1&0&0&0\\0&1&0&0\\0&0&-d&c\\0&0&b&-a\end{bmatrix}$, acting on the homogeneous coordinates of the points of $\PG(3,q^h)$ seen as column vectors by left multiplication. We see that $\psi(\Sigma)=\Sigma$, $\psi(\langle (1,\alpha,0,0)\rangle)=\langle (1,\alpha,0,0)\rangle$ and  $\psi(\langle (0,0,1,\beta)\rangle)=\langle(0,0,-d+c\beta,b-a\beta)\rangle=\langle(0,0,1,\alpha)\rangle$. Furthermore, $\psi(n_1)=n_1$ and $\psi(n_2)=n_2$. 
Since $\psi(Q_i)$ lies on the line through $\psi(Q_1)$ and $\psi(Q_2)$, we find that $\psi(Q_i)$ has coordinates of the form $\langle(1,\alpha,\xi,\xi\alpha)\rangle$ for some $\xi\in \F_{q^h}$. But since $\psi(Q_i)$ also lies on a line $\psi(n_i)$, intersecting $\Sigma$ in a subline, it contains a point of $\Sigma$ with coordinates $\langle(1,0,a',b')\rangle$ and $\langle(0,1,c',d')\rangle$. As before, this implies that $\alpha=\frac{a'\alpha-b'}{d'-c'\alpha}$.

Suppose first that $[\F_q(\alpha):\F_q]=2$. Consider a point on the line $\langle (1,\alpha,0,0),(0,0,1,\alpha)\rangle$ of the form $\langle (1,\alpha,\xi,\xi\alpha)\rangle$ with $\xi\in \F_q(\alpha)$, then we can rewrite this point as $\langle(1,\alpha,c_1\alpha+c_2,d_1\alpha+d_2)\rangle$ for some $c_1,c_2,d_1,d_2\in \F_q$. We see this point lies on the line meeting $\Sigma$ in the subline of $\Sigma$ through the points $\langle(1,0,c_2,d_2)\rangle$ and $\langle (0,1,c_1,d_1)\rangle$. This implies that the projection of $\Sigma$ from the line through $\langle (1,\alpha,0,0)\rangle,\langle(0,0,1,\alpha)\rangle$ gives rise to a linear set of size $q^2+1$ all of whose points have weight $2$, a contradiction since we assumed $L$ has weight $q^3+1$. 

So we know that $[\F_q(\alpha):\F_q]\geq 3$. It then follows from $\alpha=\frac{a'\alpha-b'}{d'-c'\alpha}$ that $a'=d'$ and $b'=c'=0$, and hence, $\psi(Q_i)$ has coordinates $\langle (1,\alpha, \lambda,\lambda\alpha)\rangle$ for some $\lambda\in \F_q$. This also shows that the sublines $n_1\cap \Sigma,n_2\cap \Sigma,\ldots,n_{q+1}\cap \Sigma$ form a regulus of $\Sigma$: they are precisely the points with homogeneous coordinates lying on the hyperbolic quadric of $\Sigma$ defined by the equation $X_0X_3=X_1X_2$.

Recall that $\psi(\Pi)$ is spanned by $\langle (1,\alpha,0,0)\rangle$ and $\langle(0,0,1,\alpha)\rangle$. If we project $\Sigma$ from $\psi(\Pi)$ onto the line $\Omega$ spanned by $\langle(0,1,0,0)\rangle$ and $\langle(0,0,0,1)\rangle$, then we see that the point $\langle(-\mu_2,\mu_1,-\mu_4,\mu_3)\rangle$ in $\Sigma$ is projected onto the point with coordinates $\langle(0,\mu_1+\mu_2\alpha,0,\mu_3+\mu_4\alpha)\rangle$. We now see that $L$ is given by the set  $\{\langle (\mu_1+\mu_2\alpha,\mu_3+\mu_4\alpha)\rangle \mid \mu_i \in \F_q,(\mu_1,\mu_2,\mu_3,\mu_4)\neq (0,0,0,0)\}$ which is of the form constructed in Theorem \ref{thm1} with $t_1=t_2=2$.
\end{proof}

\begin{remark} For linear sets of rank $4$ in $\PG(1,q^4)$, the previous proposition also follows from the classification of $\F_q$-linear blocking sets in \cite{bonoli}. \end{remark}

\subsubsection{$k=5$}

It follows from Equations \eqref{trivial2} and \eqref{trivial} that if $L$ is a linear set of rank $5$ of size $q^4+1$, then either
\begin{enumerate}
\item[(a)] $L$ contains one point of weight $4$ and the other $q^4$ points have weight one (i.e. $L$ is a club).
\item[(b)] $L$ contains one point of weight $3$, $q^2$ points of weight two and the other $q^4-q^2$ points have weight one.
\item[(c)] $L$ contains $q^2+q+1$ points of weight $2$ and the other $q^4-q^2-q$ points have weight one.
\end{enumerate}
For a linear set of type (b) in $\PG(1,q^h)$ we will make a distinction into the following two cases:
\begin{itemize}
\item[(b1)] the set of $q^2+1$ points of weight at least $2$ forms an $\F_{q^2}$-subline
\item[(b2)] the set of $q^2+1$ points of weight at least $2$ does not form an $\F_{q^2}$-subline.
\end{itemize}
Obviously, for case (b1) to occur, $h$ needs to be even.
Now consider a linear set obtained from our construction for $k=5$ with heaviest point of weight $3$: 
\[L=\{\langle(\mu_1+\mu_2\alpha,\mu_3+\mu_4\alpha+\mu_5\alpha^2)\rangle \mid \mu_i\in \F_q, (\mu_1,\mu_2,\mu_3,\mu_4,\mu_5)\neq (0,0,0,0,0) \},\] where $[\F_q(\alpha):\F_q]\geq 4$. It is clear that $\langle(1,0)\rangle$ has weight $3$ and $\langle (0,1)\rangle,(1,1)\rangle, \langle (1,\alpha)\rangle$ have weight $2$. In order for these four points to lie on an $\F_{q^2}$-subline, $\alpha$ needs to be in $\F_{q^2}$, a contradiction since we asked that $[\F_q(\alpha):\F_q]\geq 4$. This implies that a linear set of type (b1) can never be obtained from our construction.
On the other hand, we will show in Proposition \ref{alles} that linear sets of rank $5$ of size $q^4+1$ of type (b2) always arise from our construction.
\begin{remark} It is worth noting that it is possible to describe the full projective line $\PG(1,q^4)$ as a linear set of rank $5$ with exactly $q^2+q+1$ points of weight $2$ and all others of weight one (i.e. of type (c)). Equivalently one can construct a subspace of dimension $4$ intersecting a Desarguesian $3$-spread in $\PG(7,q^4)$ in exactly $q^2+q+1$ lines and $q^4-q^2-q$ points. This latter point of view also makes it easy to see that every linear set of rank $5$ in $\PG(1,q^4)$ necessarily is the set of all $q^4+1$ points of $\PG(1,q^4)$. Now consider  $L=\{\langle(\lambda_1+\lambda_2\alpha+\lambda_3 \alpha^2,\lambda_4+\lambda_1\alpha+\lambda_5\alpha^2)\rangle\mid \lambda_i\in \F_q, (\lambda_1,\lambda_2,\lambda_3,\lambda_4,\lambda_5)\neq (0,0,0,0,0)\}$, then $L$ is a linear set of rank $5$ and it is easy to check that $L$ does not contain points of weight higher than two. It follows that there are precisely $q^2+q+1$ points of weight $2$ (of which $\langle(1,0)\rangle,\langle(0,1)\rangle$ and $\langle (1,1)\rangle$ are three).

\end{remark}

\begin{proposition}\label{alles} Let $L$ be a linear set of rank $5$ with size $q^4+1$ in $\PG(1,q^h)$, containing one point of weight $3$ and $q^2$ points of weight $2$ in $\PG(1,q^h)$. If the points of weight at least $2$ do not form an $\F_{q^2}$-subline, then $L$ can be obtained from the construction of Theorem \ref{thm1}.
\end{proposition}

\begin{proof} Let $L$ be a linear set of rank $5$ in $\PG(1,q^h)$ with size $q^4+1$, containing one point of weight $3$ and $q^2$ points of weight $2$. Then $L$ can be obtained as the projection of a subgeometry $\Sigma\cong \PG(4,q)$ from a plane $\Pi$ onto a line $\Omega$ disjoint from $\Pi$ in $\PG(4,q^h)$. Moreover, there is a subplane $\mu$ of $\Sigma$ such that $\langle \Pi,P\rangle\cap \Sigma=\mu$ for a point $P$ on $\Omega$, and $q^2$ sublines, say $\ell_1,\ldots,\ell_{q^2}$ in $\Sigma$ and points $P_1,\ldots,P_{q^2}$ on $\Omega$ such that $\langle \Pi,P_i\rangle \cap \Sigma=\ell_i$. The sublines $\ell_1,\ldots,\ell_{q^2}$ of $\Sigma$ are disjoint, so any two of them span a $3$-space contained in $\Sigma$, which necessarily meets the subplane $\mu$ in a subline of $\Sigma$. Let $m_1$ be the subline of $\Sigma$ obtained as $\langle \ell_1,\ell_2\rangle \cap \mu$.

Now suppose that all sublines $\ell_i$ are contained in $\langle \ell_1,\ell_2\rangle$. Then the set of points $\{P_1,\ldots,P_{q^2},P\}$ is closed under taking $\F_q$-sublines, so it forms an $\F_{q^2}$-subline (see e.g. \cite[Theorem 1.5]{rottey}), a contradiction.

So we can consider one of the sublines $\ell_i$, say $\ell_3$, that is not contained in $\langle \ell_1,\ell_2\rangle$, and consider the subline $m_2=\langle \ell_1,\ell_3\rangle\cap \mu$ of $\Sigma$. Let $R$ be the point $m_1\cap m_2$ which is contained in $\mu$. Let $n_1$ be the unique line through $R$ meeting $\ell_1$ and $\ell_2$ and let $S_1=\ell_1\cap n_1$. Let $T\neq R$ be a point on $m_1$, let $n_2$ be the unique line through $T$ meeting $\ell_1$ and $\ell_2$, and let $S_2=\ell_1\cap n_2$. Let $U$ be a point on $\ell_2$, different from $n_1\cap \ell_2$ and $n_2\cap \ell_2$ and let $V\neq R$ be a point on $m_2$ .
Without loss of generality, $\Sigma$ can be taken to be the canonical subgeometry of $\PG(3,q^h)$, determined by the points that have a set of homogeneous coordinates with all entries in $\F_q$. Moreover, since $\mathrm{PGL}(5,q)$ acts transitively on the frames of $\Sigma$, we can pick coordinates such that $S_1=\langle(1,0,0,0,0)\rangle$, $S_2=\langle(0,1,0,0,0)\rangle$, $R=\langle(0,0,1,0,0)\rangle$, $T=\langle(0,0,0,1,0)\rangle$, $U=\langle (1,1,1,1,0)\rangle$ and $V=\langle(0,0,0,0,1)\rangle$. It follows that $n_1$ meets $\ell_2$ in $\langle(1,0,1,0,0)\rangle$ and $n_2$ meets $\ell_2$ in $\langle(0,1,0,1,0)\rangle$.

The line of $\PG(4,q^h)$ containing the subline $\ell_1$ intersects $\Pi$ in a point $X$ with coordinates $\langle (1,\alpha,0,0,0)\rangle$ for some $\alpha\in \F_{q^h}\setminus \F_q$, the line containing the subline $m_1$ intersects $\Pi$ in a point $Y$ with coordinates $\langle (0,0,1,\beta,0)\rangle$ for some $\beta\in \F_{q^h}\setminus \F_q$ and the line containing $m_2$ intersects $\Pi$ in a point $Z$ with coordinates $\langle (0,0,\gamma,0,1)\rangle$ for some $\gamma\in \F_{q^h}\setminus \F_q$.

Note that the $3$-dimensional subgeometry defined by $\ell_1$ and $\ell_2$ extends to a $3$-dimensional subspace of $\PG(4,q^h)$ intersecting $\Pi$ in the line $XY$: if this space would contain $\Pi$, $P_1$ would be a point of weight $4$, a contradiction. This implies that the line containing the subline $\ell_2$ meets $\Pi$ in a point on the line $XY$. In other words, the points $X=\langle (1,\alpha,0,0,0)\rangle,Y=\langle (0,0,1,\beta,0)\rangle,\langle(1,0,1,0,0)\rangle$ and $\langle(0,1,0,1,0)\rangle$ are co-planar which forces $\beta=\alpha$. As in the proof of Proposition \ref{prop1}, expressing that the line containing $\ell_j$ meets $\Pi$ in a point on $XZ$ forces $\gamma$ to be of the form $\frac{a\alpha+b}{c\alpha+d}$ for some $a,b,c,d\in \F_q$ with $ad-bc\neq 0$ (see also \cite[Theorem 2.4]{maartenenik} for a general argument). Hence $Z=\langle (0,0,a\alpha+b,0,c\alpha+d)\rangle$. Note that if $d=0$, then the line $YZ$ intersects the subgeometry $\Sigma$ in the point $\langle(0,0,a\alpha,-b\alpha,c\alpha)\rangle=\langle(0,0,a,-b,c)\rangle$, a contradiction since $\Pi$ is disjoint from $\Sigma$. Hence, $d\neq 0$. Now consider the projection of $\Sigma$ from $\Pi$ onto the line through the points $\langle (0,1,0,0,0)\rangle$ and $\langle (0,0,1,0,0)\rangle$: the point $\langle(\lambda_1,\lambda_2,\lambda_3,\lambda_4,\lambda_5)\rangle$ of $\Sigma$ is projected onto the point $\langle(0,\lambda_2-\lambda_1\alpha,\lambda_3+\frac{\lambda_4}{\alpha}-\frac{\lambda_5(a\alpha+b)}{c\alpha+d},0,0)\rangle$. Hence, we find that our linear set is equivalent to the set of points $\{\langle (\lambda_2-\lambda_1\alpha,(\lambda_3c-a\lambda_5) \alpha^2+(-\lambda_4c-b\lambda_5+d\lambda_3)\alpha-\lambda_4d)\rangle \mid \mu_i \in \F_q,(\lambda_1,\lambda_2,\lambda_3,\lambda_4\lambda_5)\neq (0,0,0,0,0)\}$ in $\PG(1,q^h)$. Since $d\neq 0$ and $ad-bc\neq 0$, this set equals $\{\langle (\mu_1+\mu_2\alpha,\mu_3+\mu_4\alpha+\mu_5\alpha^2)\rangle \mid \lambda_i \in \F_q,(\mu_1,\mu_2,\mu_3,\mu_4,\mu_5)\neq (0,0,0,0,0)\}$ which was constructed in Theorem \ref{thm1}.

\end{proof}

\begin{remark} For linear sets of rank $5$ in $\PG(1,q^5)$ we can say more. In \cite{maartenenik} the authors determine the possible weight distributions of  linear sets in $\PG(1,q^5)$; in particular, it follows that possibility (c) does not occur for linear sets of rank $5$ in $\PG(1,q^5)$. Moreover, we have seen in Subsection \ref{trace} that all linear sets in $\PG(1,q^5)$ of type (a) can be obtained from our construction. Since $5$ is odd, there are no linear sets of type (b1). Combined with Proposition \ref{alles} this shows that all $\F_q$-linear sets of size $q^4+1$ in $\PG(1,q^5)$ arise from our construction.
\end{remark}

\begin{remark} Consider two linear sets, say $L_1$ and $L_2$, of rank $k$ in $\PG(1,q^h)$ obtained as the projection of a subgeometry $\Sigma\cong \PG(k-1,q)$ from a subspace $\Pi_1$ and $\Pi_2$ resp. If there is an element of $\PGammaL(k,q^h)$ such that $\phi(\Sigma)=\Sigma$ and $\phi(\Pi_1)=\Pi_2$, then $L_1$ and $L_2$ are $\PGammaL(2,q^h)$-equivalent. We have used this idea in Proposition \ref{prop1} to show that all rank $4$ linear sets of size $q^3+1$ with $(q+1)$ points of weight $2$ can be obtained as in Theorem \ref{thm1}. But, as was pointed out in \cite{corrado}, this condition is not necessary: it is possible for two linear sets $L_1$ and $L_2$ on a line to be equivalent, even though there is no element of $\PGammaL(k,q^h)$ stabilising $\Sigma$ and mapping $\Pi_1$ onto $\Pi_2$.
This is one of the reasons why the equivalence problem for linear sets is difficult, even just for the case of clubs of rank $k<h$ (see Open Problem (C)).

\end{remark}

\subsection{Linear sets of small size in projective spaces of dimension at least two.}
Note that in Construction \ref{cons1} we partitioned the set $\{\textbf{e}_\textbf{1},\dots, \textbf{e}_\textbf{k}\}$ in two parts which in turn give us two points that form the line $\Omega$ onto which we are projecting. We can generalise the same idea to higher dimensional spaces by increasing the number of partitions we make. To preserve the polynomial behaviour as described in Lemma \ref{lemnew}, in this case, we will have to restrict the size of the partitions in such a way that sum of the sizes of any two of them is bounded above by $s+1$. This construction embeds the linear set on a line constructed in the previous section in a higher dimensional space.  We refer to Remark \ref{rem} for a more detailed description of the geometric construction.

\begin{lemma}\label{lem4} Let $L$ be the set of points with coordinates $$\langle (f_1(\alpha),f_2(\alpha),\dots,f_{l+1}(\alpha))\rangle $$ in $\PG(l,q^h)$
where $\alpha$ is an element of $\F_{q^h}\setminus \F_q$ such that $[\F_q(\alpha):\F_q]=s$, $f_i$ is a polynomial in degree at most $t_i-1$ over $\mathbb{F}_q$ such that for all $i\neq j$, $t_i+t_j\leq s+1$, and $1\le t_1, t_2, \ldots, t_{l+1}\le h$.

There is a bijection between the set of points of $L$ and  the set $$S=\{(f_1,f_2,\dots,f_{l+1})|\deg{(f_i)}\leq t_i-1,1\le t_1\le t_2\dots\le t_{l+1},$$$$(f_1,f_2,\dots,f_{l+1})\text{ is in reduced form}\}$$\end{lemma}

\begin{proof}Without loss of generality, we can rearrange the coordinates for $L$ such that $t_1\leq t_2\leq \ldots\leq t_{l+1}$.We proceed by induction on $l$ to compute the size of $S$. 
As in Lemma \ref{lem2}, we see that every point $P$ of the form $\langle f_1(\alpha)\textbf{e}_{\textbf{1}\textbf{,}{\textbf{t}_\textbf{1}}}+\dots+f_{l+1}(\alpha)\textbf{e}_{{\textbf{l}\textbf{+}\textbf{1}}\textbf{,}{\textbf{t}_{\textbf{l}\textbf{+}\textbf{1}}}}\rangle$ can be uniquely represented in its reduced form.
If  two reduced forms $(f_1,\dots,f_{l+1})$ and $(f_1',\dots,f_{l+1}')$ represent the same point in the constructed linear set, then we have that $f_if_j'=f_jf_i'$ for all $i\neq j$. Let $i_0$ be the value such that $f_{i_0}\neq 0$ but $f_i=0$ for all $i<i_0$. It is clear, since  $(f_1,\dots,f_{l+1})$ and $(f_1',\dots,f_{l+1}')$ represent the same point, that $f'_{i_0}\neq 0$ and $f'_i=0$ for all $i<i_0$. Since  $(f_1,\dots,f_{l+1})$ and $(f_1',\dots,f_{l+1}')$ are in reduced form, $f_{i_0}$ and $f'_{i_0}$ are monic.
From $f_{i_0}f'_j=f'_{i_0}f_j$, and $\gcd(f_{i_0},f_j)=\gcd(f'_{i_0},f'_j)=1$ it follows by Lemma \ref{lemnew} that $f_j=f'_j$, for all $j$.
\end{proof}

\begin{lemma}\label{newlem2}
The number of elements in the set $$S=\{(f_1,f_2,\dots,f_{l+1})|\deg{(f_i)}\leq t_i-1,1\le t_1\le t_2\dots\le t_{l+1},$$$$(f_1,f_2,\dots,f_{l+1})\text{ is in reduced form}\}$$ is $q^{k-1}+q^{k-2}+\dots+q^{k-l}+1$, where $k=\sum_{i=1}^{l+1}t_i$.
\end{lemma}

\begin{proof}By Lemma \ref{newlem1} we have that the size of $S$ for $l=1$ with deg$(f_1)\le t_1-1$ and deg$(f_2)\le t_2-1$ is $q^{t_1+t_2-1}+1$. This provides the base case. Now, assume that for all $l'$ with $1\leq l'\leq l-1$, the number of tuples in 
$$S'=\{(f_1,f_2,\dots,f_{l'+1})|\deg{(f_i)}\leq t_i-1,1\le t_1\le t_2\dots\le t_{l'+1},$$$$(f_1,f_2,\dots,f_{l'+1})\text{ is in reduced form}\}$$
is equal to $q^{k'-1}+q^{k'-2}+\dots+q^{k'-l'}+1$, where $k'=\sum_{i=1}^{l'+1}t_i$.

We first count the number of tuples in $S$ that have $f_1$ equal to zero and see that the number of such tuples is the same as the number of points in $S'$ with $l'=l-1$ and $k'=k-t_1$ which has $$q^{k-t_1-1}+q^{k-t_1-2}+\dots+q^{k-t_1-(l-1)}+1$$ elements by our induction hypothesis.

We will now count the number of tuples in $S$ for which $f_1$ is non-zero and monic, in a similar fashion as we have done in Lemma \ref{newlem1}. Let $\bar{m}$ be a vector of length $l$ whose entries are $(m_2,\ldots,m_{l+1})$ and let $P_{n,\leq \bar{m}}$, denote all $(l+1)$-tuples $(f_1,f_2,\dots,f_{l+1})$ where $\deg(f_1)=n$, $\deg(f_i)\leq m_i$ for all $i\in\{2,\ldots,l+1\}$.
Let $R_{n,\leq \bar{m}}$ denote the number of $(l+1)$-tuples in $P_{n,\leq \bar{m}}$ in reduced form.
We see that the number of tuples in $S$ for which $f_1$ is non-zero equals the size of $\cup_{i=0}^{t_1-1}R_{i,\leq \bar{t}-\bar{1}}$, where $\bar{t}=(t_2,\ldots,t_{l+1})$ and $\bar{1}$ is the all-one vector.

We see that there are $q^{n_2+\dots+n_{l+1}+l}$ tuples in $R_{0,\leq \bar{n}}$, where $\bar{n}=(n_2,\ldots,n_{l+1})$.
We have $q^{n_2+\dots+n_{l+1}+l+1}$ tuples in $P_{1,\leq \bar{n}}$. Of those tuples, in order to calculate the size of $R_{1,\le \bar{n}}$, we have to exclude the tuples that arise from a tuple with $\deg(f_1)=0$ and $\deg(f_j)\leq n_j-1$ for all $j\in \{2,\ldots,l+1\}$. Each of those gives rise to $q$ different tuples with $\gcd(f_1,\dots,f_{l+1})\ne1$ in $P_{1,\leq \bar{n}}$. Hence, we find that
$|R_{1,\leq \bar{n}}|=q^{n_2+\dots+n_{l+1}+l+1}-q\cdot q^{n_2+\dots+n_{l+1}}$. We claim that 
$$|R_{r,\leq \bar{n}}|=q^{n_2+\dots+n_{l+1}+l+s}- q^{n_2+\dots+n_{l+1}+s}$$
and have seen that the statement holds for $r=1$. Now assume that for a certain $m$, our claim holds for all $1\le r\le m \le t_1-2$ and $\deg(f_j)\le n_j$, $(m\le n_j\le t_j-1, j\in\{2,\dots,l+1\})$. 

We know that there are $q^{n_2+\dots+n_{l+1}+l+m+1}$ tuples $(f_1,\dots,f_{l+1})$ in $P_{m+1,\leq \bar{n}}$. We see that all tuples in $R_{m,\leq \bar{n}-\bar{1}}$, each give rise to $q$ tuples in $P_{m+1,\leq \bar{n}}$, and in more general, for all $i\in \{1,\ldots,m+1\}$, all tuples in $R_{m+1-i,\bar{n}-i\bar{1}}$ give rise to $q^i$ tuples in $P_{m+1,\leq \bar{n}}$. We conclude that there are
\begin{align*}
&q^{n_2+\dots+n_{l+1}+l+m+1}\\
&-q^{-l}(q^{n_2+\dots+n_{l+1}+l+m+1}-q^{n_2+\dots+n_{l+1}+m+1})\\
&-q^{-2l}(q^{n_2+\dots+n_{l+1}+l+m+1}-q^{n_2+\dots+n_{l+1}+m+1})\\
\vdots\\
&-q^{-ml}(q^{n_2+\dots+n_{l+1}+l+m+1}-q^{n_2+\dots+n_{l+1}+m+1})\\
&-q^{-(m+1)l}(q^{n_2+\dots+n_{l+1}+l+m+1})\\
&= q^{n_2+\dots+n_{l+1}+l+m+1}-q^{n_2+\dots+n_{l+1}+m+1}
\end{align*} tuples in $R_{m+1,\leq \bar{n}}$. 

Now recall that $|S|=q^{k-t_1-1}+q^{k-t_1-2}+\dots+q^{k-t_1-(l-1)}+1+\cup_{i=0}^{t_1-1} |R_{i,\leq \bar{t}-\bar{1} }|$.
Since $\cup_{i=0}^{t_1-1} R_{i,\leq \bar{t}-\bar{1} }$ has size 
$q^{k-t_1}+(q^{k-t_1+1}-q^{k-t_1-l+1})+\ldots+(q^{k-2}-q^{k-2-l})+(q^{k-1}-q^{k-1-l})$,
we have that $|S|=q^{k-1}+q^{k-2}+\dots+q^{k-l}+1$ as claimed.
\end{proof}

\begin{theorem}\label{main}Let $L$ be the set of points with coordinates $$\langle (f_1(\alpha),f_2(\alpha),\dots,f_{l+1}(\alpha))\rangle,$$ in $\PG(l,q^h)$
where $\alpha$ is an element of $\F_{q^h}\setminus \F_q$ such that $[\F_q(\alpha):\F_q]=s$, $f_i$ is a polynomial over $\mathbb{F}_q$ of degree at most $t_i-1$ such that for all $i\neq j$, $t_i+t_j\leq s+1$.

	Then $L$ is an $\F_q$-linear set of rank $k=\sum_{i=1}^{l+1}t_i$ with $q^{k-1}+q^{k-2}+\dots+q^{k-l}+1$ points.
\end{theorem}	

\begin{proof} The set of vectors of the form $(f_1(\alpha),f_2(\alpha),\dots,f_{l+1}(\alpha))$ is closed under addition and $\F_q$-multiplication, and has size $q^{t_1+t_2+\ldots+t_{l+1}}=q^k$. Hence, the set of points with coordinates of the form $\langle (f_1(\alpha),f_2(\alpha),\dots,f_{l+1}(\alpha))\rangle$ is an $\F_q$-linear set of rank $k$.
Without loss of generality we may assume that $t_1\le t_2\le\dots\le t_{l+1}$. Lemma \ref{lem4} provides a bijection between the points in $L$ and $S$, where $$S=
\{(f_1,f_2,\dots,f_{l+1})|\deg{f_i}\leq t_i-1,1\le t_1\le t_2\dots\le t_{l+1},$$$$ (f_1,f_2,\dots,f_{l+1})\text{ is in reduced form}\}$$ and Lemma \ref{newlem2} shows that the number of elements in $S$ is equal to $q^{k-1}+q^{k-2}+\dots+q^{k-l}+1$. Hence $L$ has size $q^{k-1}+q^{k-2}+\dots+q^{k-l}+1$.
\end{proof}

\begin{corollary} \label{thm2}
	We can construct rank $k$ linear sets of size $q^{k-1}+q^{k-2}+\dots+q^{k-l}+1$ in $\PG(l,q^h)$, for all $l+1\leq k\le (l+1)\frac{h+1}{2}$ if $h$ is odd and $l+1\leq k\le 1+(l+1)\frac{h}{2}$ if $h$ is even.
\end{corollary}

\begin{proof} By taking a generator $\alpha$ in $\F_{q^h}^*$, we have $s=h$. We have to write $k=t_1+\ldots+t_{l+1}$ for $t_i$ with $t_i+t_j\leq h+1$ for $i\neq j$. Hence, if $h$ is odd, we can take $t_1=t_2=\ldots=t_{l+1}=\frac{h+1}{2}$ and find that $k=(l+1)\frac{h+1}{2}$. For $l+1\leq k<(l+1)\frac{h+1}{2}$ we just need to lower the partition sizes $t_i$ accordingly as long as $t_i\geq 1$.
If $h$ is even, then we can take $t_1=t_2=\ldots=t_l=\frac{h}{2}$ and $t_{l+1}=\frac{h}{2}+1$, to find $k=(l+1)\frac{h}{2}+1$. As before, for $l+1\leq k<(l+1)\frac{h}{2}+1$ we just need to adjust the partition sizes $t_i$ accordingly.
\end{proof}

\begin{remark} \label{remarkweights}The weight of different points can be calculated in a similar fashion as in Theorem \ref{thm1}. Just as before, the tuples in the proof of Theorem \ref{thm2}  which were being counted repeatedly represent points with weight more than 1. Table 2 summarises the conditions on deg$(f_j)=d_j$, $(j=2,\dots,l+1)$ to obtain points of different weights, depending on deg$(f_1)=d_1$. 
Recall that the weight of the points range from 1 to $t_1$ when $f_1\ne0$. Summing the numbers in each column we find that for $f_1\ne0$, the total number of points having weight $w\in\{1,\dots,t_1-2\}$ is \begin{equation}\sum_{i=w}^{t_1}(q^{k-(w-1)l-i}-q^{k-wl-i})-\sum_{i=w+1}^{t_1-1}(q^{k-wl-i}-q^{k-(w+1)l-i}),\label{eqweights}\end{equation}	the total number of points having weight $w=t_1-1$ is \begin{equation}\sum_{i=w}^{t_1}(q^{k-(w-1)l-i}-q^{k-wl-i}),\label{eqweights2}\end{equation}
and the total number of points having weight $w=t_1$ is $q^{k-(t_1-1)l-t_1}$.
Note that this table identifies the weight of points for which $f_1\ne0$ but recall that the case $f_1=0$ is equivalent to a $(k-t_1)$-rank linear set in $\PG(l-1,q^h)$ defined in the same way. So we can use the same idea to identify the weight of points for which $f_1=0$. Hence the construction gives a linear set which has $q^{k-1}+q^{k-2}+\dots+q^{k-l}$ points whose weights range from $1,2,\dots,t_{l}$, having at least one point in each category, and exactly one point of weight $t_{l+1}$ if $t_l<t_{l+1}$.  See Example \ref{exam} for an explicit example.
\end{remark}	
	
\begin{table}
	\resizebox{\textwidth}{!}{
	\begin{tabular}{|c||c|c|ccc|c|c|}
		\hline 
		
		$(d_1)$& weight 1 & weight 2   &&$\dots$ & & weight $t_1-1$  & weight $t_1$  \\ 
		\hline
		\hline 
		$t_1-1$& $\{d_j\le t_j-1\}$  &0  &&$\dots$& &0  &0  \\ 
		&&&&&&&\\
		\hline 
		$t_1-2$& $\{d_j\le t_j-1\}$ & $\{d_j\le t_j-2\}$ &&$\dots$&   &0  &0  \\ 
		&$-\{d_j\le t_j-2\}$&&&&&&\\
		\hline 
		$t_1-3$& $\{d_j\le t_j-1\}$ & $\{d_j\le t_j-2\}$ &&$\dots$&   &  0&  0\\ 
		&$-\{d_j\le t_j-2\}$&$-\{d_j\le t_j-3\}$&&&&&\\
		\hline 
		&&&&&&&\\
		$\vdots$& $\vdots$ & $\vdots$ && $\ddots$ &&$\vdots$ & $\vdots$  \\ 
		&&&&&&&\\
		\hline 
		1& $\{d_j\le t_j-1\}$ &$\{d_j\le t_j-2\}$  &   &$\dots$& & $\{d_j\le t_j-(t_1-1)\}$ & 0 \\ 
		&$-\{d_j\le t_j-2\}$&$-\{d_j\le t_j-3\}$&&&&&\\
		\hline 
		0& $\{d_j\le t_j-1\}$ & $\{d_j\le t_j-2\}$ & &$\dots$& & $\{d_j\le t_j-(t_1-1)\}$  & $\{d_j\le t_j-t_1\}$ \\ 
		&$-\{d_j\le t_j-2\}$&$-\{d_j\le t_j-3\}$&&&&$-\{d_j\le t_j-t_1\}$&\\
		\hline 
	\end{tabular}
}
\smallskip
\caption{Conditions on the degrees of $f_i$ for obtaining points of different weights when $f_1$ is non-zero.}
\end{table} 

\begin{remark}\label{rem} It is not too hard to check that the linear set $L$ can be obtained as the projection of $\Sigma=\PG(k-1,q)$, a canonical subgeometry of $\Sigma^*=\PG(k-1,q^h)$, from $\Pi$ to $\Omega$, defined as follows:

Partition the set $\{\textbf{e}_\textbf{1},\dots,\textbf{e}_\textbf{k}\}$ into $(l+1)$ parts $A_1,A_2,\dots,A_{l+1}$ of size $1\le t_1,t_2,\dots,t_{l+1}$$\le h$ (as ordered sets) such that the sum of any 2 of them is at most $s+1$. Let 
\begin{align*}
A_1&=\{\textbf{e}_\textbf{1},\dots,\textbf{e}_{\textbf{t}_\textbf{1}}\}=\{\textbf{e}_{\textbf{1}\textbf{,}\textbf{1}},\dots,\textbf{e}_{\textbf{1}\textbf{,}{\textbf{t}_\textbf{1}}}\}\\
A_2&=\{\textbf{e}_{\textbf{t}_\textbf{1}\textbf{+}\textbf{1}},\dots,\textbf{e}_{\textbf{t}_\textbf{1}\textbf{+}\textbf{t}_\textbf{2}}\}=\{\textbf{e}_{\textbf{2}\textbf{,}\textbf{1}},\dots,\textbf{e}_{\textbf{2}\textbf{,}{\textbf{t}_\textbf{2}}}\}\\
&\vdots\\
A_{l+1}&=\{\textbf{e}_{\textbf{t}_\textbf{1}\textbf{+}\textbf{t}_\textbf{2}\textbf{+}\dots\textbf{+}\textbf{t}_\textbf{l}\textbf{+}\textbf{1}},\dots,\textbf{e}_{(\textbf{t}_\textbf{1}\textbf{+}\textbf{t}_\textbf{2}\textbf{+}\dots\textbf{+}\textbf{t}_{\textbf{l}\textbf{+}\textbf{1}})=\textbf{k}}\}=\{\textbf{e}_{{\textbf{l}\textbf{+}\textbf{1}}\textbf{,}\textbf{1}},\dots,\textbf{e}_{{\textbf{l}\textbf{+}\textbf{1}}\textbf{,}{\textbf{t}_{\textbf{l}\textbf{+}\textbf{1}}}}\}.
\end{align*}
With each $A_i, i=1,2,\dots,l+1$, associate the subspace of $\Sigma^*$  defined as
 $\pi_i=\emptyset$ when $t_i=1$ and,
\begin{align*}
\pi_i&=\langle \langle\textbf{e}_{\textbf{i}\textbf{,}\textbf{1}}-\alpha \textbf{e}_{\textbf{i}\textbf{,}{\textbf{2}}}\rangle,\langle\textbf{e}_{\textbf{i}\textbf{,}{\textbf{2}}}-\alpha \textbf{e}_{\textbf{i}\textbf{,}{\textbf{3}}}\rangle,\dots,\langle\textbf{e}_{\textbf{i}\textbf{,}{\textbf{t}_{\textbf{i}\textbf{-}\textbf{1}}}}-\alpha \textbf{e}_{\textbf{i}\textbf{,}{\textbf{t}_{\textbf{i}}}}\rangle\rangle\\
&=\langle \langle\textbf{e}_{\textbf{i}\textbf{,}\textbf{1}}-\alpha \textbf{e}_{\textbf{i}\textbf{,}\textbf{2}}\rangle,\langle\textbf{e}_{\textbf{i}\textbf{,}\textbf{1}}-\alpha^2\textbf{e}_{\textbf{i}\textbf{,}\textbf{3}}\rangle,\dots,\langle\textbf{e}_{\textbf{i}\textbf{,}\textbf{1}}-\alpha^{t_i-1}\textbf{e}_{\textbf{i}\textbf{,}{\textbf{t}_\textbf{i}}}\rangle\rangle,
\end{align*} 
if $t_i\geq 2$. 
Let $\Omega$ be the subspace $$\Omega:=\langle \langle\textbf{e}_{\textbf{1}\textbf{,}{\textbf{t}_\textbf{1}}}\rangle,\langle\textbf{e}_{\textbf{2}\textbf{,}{\textbf{t}_\textbf{2}}}\rangle,\dots,\langle\textbf{e}_{\textbf{l}\textbf{+}\textbf{1}\textbf{,}{\textbf{t}_{\textbf{l}\textbf{+}\textbf{1}}}}\rangle\rangle$$
and $\Pi$ be the $(k-l-2)$-dimensional subspace $$\Pi:=\langle\pi_1,\dots,\pi_{l+1}\rangle.$$
\end{remark}

\begin{example}\label{exam}
Consider the case $l=2,t_1=2,t_2=3,t_3=4$ in Theorem \ref{main}: take the set of points $$L=\langle(f_1(\alpha),f_2(\alpha),f_3(\alpha))\rangle$$ in $\PG(2,q^h)$, where $f_1$ is a linear, $f_2$ is a quadratic and $f_3$ is a cubic polynomial over $\mathbb{F}_q$ with $\alpha$ generating an extension of $\mathbb{F}_q$ of degree at least $t_2+t_3-1=6$. More explicitely, $L$ is described as 
\[L=\{ \langle (\mu_1+\mu_2\alpha,\mu_3+\mu_4\alpha+\mu_5\alpha^2,\mu_6+\mu_7\alpha+\mu_8\alpha^2+\mu_9\alpha^3)\rangle \mid \mu_i\in \mathbb{F}_q,\]\[(\mu_1,\mu_2,\ldots,\mu_9)\neq (0,0\ldots,0)\}.\]
This set $L$ clearly forms a rank 9 linear set in $\PG(2,q^h)$ and we have shown in Theorem \ref{main} that it has size $q^8+q^7+1$. As mentioned in Remark \ref{rem}, $L$ can be viewed as a projection of $\Sigma=\PG(8,q)$, a canonical subgeometry of $\Sigma^*=\PG(8,q^h)$, from $\Pi=\langle\langle\textbf{e}_\textbf{1}-\alpha\textbf{e}_\textbf{2}\rangle,\langle\textbf{e}_\textbf{3}-\alpha\textbf{e}_\textbf{4}\rangle,\langle\textbf{e}_\textbf{4}-\alpha\textbf{e}_\textbf{5}\rangle,\langle\textbf{e}_\textbf{6}-\alpha\textbf{e}_\textbf{7}\rangle,\langle\textbf{e}_\textbf{7}-\alpha\textbf{e}_\textbf{8}\rangle,\langle\textbf{e}_\textbf{8}-\alpha\textbf{e}_\textbf{9}\rangle\rangle$ to $\Omega=\langle\langle\textbf{e}_\textbf{2}\rangle,\langle\textbf{e}_\textbf{5}\rangle,\langle\textbf{e}_\textbf{9}\rangle\rangle$. Counting the number of points of $L$ in a particular weight category as in Remark \ref{remarkweights}, we get that, when $f_1\ne0$, there are 
\begin{align*}
	q^8+q^7-q^6-q^5 &\text{ points of weight } 1,\\
	q^5 &\text{ points of weight } 2,
\end{align*}
when $f_1=0,f_2\ne0$, there are 
\begin{align*}
q^6-q^4 &\text{ points of weight } 1,\\
q^4-q^2 &\text{ points of weight } 2,\\
q^2 &\text{ points of weight } 3
\end{align*}
and when $f_1=f_2=0,f_3$ has to be non-zero and this represents a unique point of weight 4. So, in total, we find $q^8+q^7-q^5-q^4$ points of weight $1$, $q^5+q^4-q^2$ points of weight $2$, $q^2$ points of weight $3$, and one point of weight $4$.
\end{example}

\subsection{Blocking sets of non-R\'edei type from Theorem \ref{main}}

Inspired by the definition of R\'edei type blocking sets, we could call a linear set of rank $k$ in $\PG(2,q^h)$ such that there is a line intersecting the linear set in a linear set of rank $k-1$, a {\em linear set of R\'edei type}. It is easy to see that the linear set of Theorem \ref{main} in $\PG(2,q^h)$ is then of R\'edei type if and only if one of the three partitions of $\{\textbf{e}_\textbf{1},\dots,\textbf{e}_\textbf{k}\}$ is of size one (i.e. if one of the $t_i$'s equals $1$), or equivalently, if one of the coordinates is defined by a variable ranging in $\F_q$. Consequently, the linear set constructed in Example \ref{exam} is of non-R\'edei type.

Now consider a linear blocking set arising from Theorem \ref{main}, that is, a linear set of rank $k=h+1$ constructed in $\PG(2,q^h)$ in Theorem \ref{main}. We see that $h+1=t_1+t_2+t_3$ for some $1\leq t_1\leq t_2\leq t_3$ and we have seen it is of R\'edei type if and only if $t_1=1$. Now we also see that if there are two different R\'edei lines then two of the partitions must have size one, and hence $t_1=t_2=1,t_3=h-1$.
It follows, considering Subsection \ref{trace}, that this linear set intersects its R\'edei line in a linear set defined by the Trace map. It has been proven in general that a linear blocking set with at least two R\'edei lines arises from the Trace map in \cite[Theorem 5]{polverino}. 

Hence, in order to have linear blocking sets of non-R\'edei type arising from our construction, we simply need to take the size of all the partitions to be at least two. This implies that $h+1\geq 2+2+2=6$. Note that this agrees with the fact that there are no non-R\'edei type blocking sets of size $q^h+q^{h-1}+1$ for $h=2,3,4$; there do exist non-R\'edei type blocking sets in $\PG(2,q^4)$, but the smallest one of those has size $q^4+q^3+q^2+1$ (see \cite{bonoli}).

\subsection{Open Problems}

In this paper, we have constructed a large class of $\F_q$-linear sets of rank $k$ in $\PG(l,q^h)$ of size $q^{k-1}+q^{k-2}+\ldots+q^{k-l}+1$. This leads to the some questions for future research which were not (or only partially) addressed in this paper.

\begin{enumerate}[(A)]
\item Linear sets of rank $h$ in $\PG(1,q^h)$ have been studied quite intensively over the past years. It is well-known that every such linear set is equivalent to one of the form $\{\langle x,f(x)\rangle:x\in \F_{q^h}\setminus\{0\}\}$, where $f$ is a linearised polynomial. From this point of view, finding linear sets of smallest possible size comes down to finding linearised polynomials $f$ such that $|Im(f(x)/x)|$ is as small as possible. In \cite[Open Problem 3]{carlitz} the authors ask to classify such polynomials $f$ (possibly for small values of $h$) and to find new examples. Even though we know that all the examples for $k=h$ and $l=1$ constructed in this paper give rise to such polynomials, we were unfortunately not able to write them down explicitely. 

\item In \cite[Theorem 4.4]{vdv}, it was shown that for a linear set $L$ of rank $k$ in $\PG(l,q^h)$, the  bound $|L|\geq q^{k-1}+q^{k-2}+\cdots+q^{k-l}+1$ holds when there is a hyperplane $\Pi$ of $\PG(l,q^h)$ such that $|L\cap \Pi|=\frac{q^{l}-1}{q-1}$ and these intersection points span $\Pi$ (*). In the same paper (see \cite[Proposition 4.2]{vdv}), the authors construct an example of a linear set of rank $k$ in $\PG(2,q^k)$ with $|L|\geq q^{k-1}+q^{k-2}+\cdots+q^{k-l}+1$ and satisfying the mentioned condition (*). An example with the same weight distribution as theirs can be found by taking our example with $l=2$, $t_1=1,t_2=1,t_3=k-2,s=k$. It is easy to see that this generalises to create an example meeting the lower bound and condition (*) for all $2\leq l\leq k-1$, when we take $t_1=t_2=\ldots=t_{l}=1,t_{l+1}=k-l,s=k$. Furthermore, there are other choices for $t_i$ for which our constructed linear set meets the lower bound and satisfies (*). E.g, consider a linear set of rank $t_1+t_2+t_3=k$ in $\PG(2,q^h)$ obtained from our construction. It is not too hard to show that the line through the points $\langle (1,f_2(\alpha),0)\rangle$ and $\langle (1,0,f_3(\alpha))\rangle$, where $\deg(f_2)=t_2-1$ and $\deg(f_3)=t_3-1$ and $\gcd(f_2,f_3)=1$ is a $(q+1)$-secant whenever $(t_1-1)+(t_2-1)+(t_3-1)=k-3<s\leq h$. Furthermore, note that an $\F_q$-linear set of rank at least $h+3$ in $\PG(2,q^h)$ necessarily intersects every line of an $\F_q$-linear set of rank at least $3$ (i.e. in at least $q^2+1$ points), so it can never satisfy condition (*). We see that it is not because an $\F_q$-linear set of rank $k$ meets the lower bound for its size that there is a secant line to it that meets the theoretical lower bound of $q+1$. 

We would like to repeat here that the authors of \cite{vdv} left it as an open problem to replace the condition (*) by a condition that is less restrictive. We believe the lower bound $|L|\geq q^{k-1}+q^{k-2}+\cdots+q^{k-l}+1$ holds for all linear sets of rank $k$ in $\PG(l,q^h)$ where $h$ is prime and $k< h+3$.

\item We have only very briefly touched upon the subject of equivalence of our constructed linear sets (e.g. in Subsection \ref{k23}). In general, the equivalence problem for linear sets is hard, and only partial results have been obtained so far; most of those results again deal with the case $l=1,k=h$ \cite{classes,bence2,corrado}. The problem of finding the equivalence classes of linear sets constructed by us in this paper is open; in particular, is it possible to find equivalent linear sets in $\PG(l,q^h)$ of the same rank for different $t_i$'s? 
It is worth remembering that the weight of a point in a linear set $L$ is only defined when the subspace $U$ defining the linear set is fixed. This also means that it is possible to have one linear set $L_U=L_V$ such that the weight of a point of $L_U$ is different from the weight of the point in $L_V$ (this issue is known for quite some time now, see \cite[Remark]{fancsali} for an early reference to it). 

\item We have shown in Proposition \ref{prop1}, that for $k=4$, our construction contains all possibilities for an $\F_q$-linear set of rank $k$ in $\PG(1,q^h)$. For $k=5$, in Proposition \ref{alles}, we needed to exclude the possibility that that the linear set contained an $\F_{q^2}$-subline but could still draw the same conclusion. However, we do not believe that for general $k$, our construction will construct all possible examples of linear sets of minimum size. But it would be interesting to find explicit examples of linear sets that meet the lower bound and cannot be obtained from our construction -- or to disprove our belief!
\end{enumerate}

\vspace{1cm}
\noindent
{\bf Acknowledgement:} The authors would like to thank the referees for their detailed reading and suggestions. In particular, we thank one of the referees for pointing out the link with references \cite{benjamin} and \cite{corteel}.

\vspace{1cm}

\noindent
{\bf Address of the authors:}\\
Dibyayoti Jena, Geertrui Van de Voorde\\
School of Mathematics and Statistics --{\em Te Kur\=a Pangarau}\\
University of Canterbury -- {\em Te Whare W\=ananga o Waitaha}\\
Private bag 4800\\
8140 Christchurch --{\em \=Otautahi}\\
New Zealand --{\em Aotearoa}

\vspace{0.5cm}
\noindent
dibyayoti.jena@pg.canterbury.ac.nz, geertrui.vandevoorde@canterbury.ac.nz


\begin{thebibliography}{99}

\bibitem{simeon}S. Ball. The number of directions determined by a function over a finite field. {\em J. Combin. Theory, Ser. A} {\bf 104} (2003), 341--350.

\bibitem{bbb}
A. Blokhuis, S. Ball, A. Brouwer, L. Storme, and T. Sz\H{o}nyi. On the number of slopes determined by a function on a finite field. {\em J. Combin. Theory, Ser. A} {\bf 86} (1999), 187--196.

\bibitem{benjamin}A. Benjamin and C. Bennett. The probability of relatively prime polynomials. {\em Math. Mag.} {\bf 80 (3)}(2007), 196--202.
\bibitem{blokhuis} A. Blokhuis and M. Lavrauw. Scattered Spaces with Respect to a Spread in $\PG(n,q)$. {\em Geom. Dedicata} {\bf 81} (2000), 231--243.

\bibitem{bonoli}
{G. Bonoli} and {O. Polverino.}
$\mathbb{F}_q$-Linear blocking sets in $\PG(2,q^4)$.
\textit{Innov. Incidence Geom.} \textbf{2} (2005), 35--56.

\bibitem{corteel} S. Corteel, C. Savage, H. Wilf, and D. Zeilberger. A Pentagonal Number Sieve. {\em J. Combin. Theory, Ser. A} {\bf 82 (2)} (1998), 186--192.


\bibitem{classes} B. Csajb\'ok, G. Marino, and O. Polverino. Classes and equivalence of linear sets in $\PG(1,q^n)$. {\em J. Combin. Theory, Ser. A} {\bf 157} (2018), 402--426.
\bibitem{carlitz} B. Csajb\'ok, G. Marino, and O. Polverino. A Carlitz type result for linearized polynomials. {\em Ars Math. Contemp.} {\bf16(2)} (2019), 585--608.
\bibitem{bence} B. Csajb\'ok, G. Marino, O. Polverino, and C. Zanella. A new family of MRD-codes. {\em Linear Algebra Appl.} {\bf 548} (2018), 203--220.

\bibitem{bence2} B. Csajb\'ok, G. Marino, and F. Zullo. New maximum scattered linear sets of the projective line. {\em Finite Fields Appl.} {\bf 54} (2018), 133--150.
	
\bibitem{corrado} B. Csajb\'ok and C. Zanella. On the equivalence of linear sets. {\em Des. Codes Cryptogr.} {\bf 81} (2016), 269--281.

\bibitem{vdv}
{J. De Beule} and {G. Van De Voorde.}
The minimum size of a linear set.
\textit{J. Combin. Theory, Ser. A} \textbf{164} (2019), 109--124.

\bibitem{maarten} M. De Boeck and G. Van de Voorde. A linear set view on KM-arcs. {\em  J. Algebraic Combin.} {\bf 44} (2016), 131--164.

\bibitem{maartenenik} M. De Boeck and G. Van de Voorde. The weight distributions of linear sets in $\PG(1,q^5)$. arXiv:2006.04961

\bibitem{fancsali}Sz. Fancsali and P. Sziklai. Description of the clubs. {\em Ann. Univ. Sci. Budapest. E\"otv\"os Sect. Math.} {\bf 51} (2008), 141--146.

\bibitem{olgamichel} M. Lavrauw and O. Polverino. {\em  Finite Semifields}. Current research topics in Galois geometries. Nova Academic Publishers, 2011. 

\bibitem{lavrauw} M.
Lavrauw.
{\em Scattered spaces in Galois geometry}. Contemporary developments in finite fields and applications, 195--216, World Sci. Publ., Hackensack, NJ, 2016.

\bibitem{wij} M. Lavrauw and G. Van de Voorde. {\em  Field reduction and linear sets in finite geometry.} Topics in finite fields, 271--293, Contemp. Math., {\bf 632}, Amer. Math. Soc., Providence, RI, 2015.

\bibitem{lunardon}
{G. Lunardon} and {O. Polverino.}
Translation ovoids of orthogonal polar spaces.
\textit{Forum Math.} \textbf{16 (5)} (2004), 663--669.
	
\bibitem{polverino}
{G. Lunardon} and {O. Polverino.}
Blocking sets of size $q^t+q^{t-1}+1$.
\textit{J. Combin. Theory, Ser. A} \textbf{90} (2000), 148--158.

\bibitem{polito}
P. Polito and O. Polverino. On small blocking sets. {\em Combinatorica} {\bf 18 (1)} (1998), 133--137.

\bibitem{olga}
O.~Polverino.
\newblock Linear sets in finite projective spaces.
\newblock {\em Discrete Math.} {\bf 310 (22)} (2010), 3096--3107.

\bibitem{zullo}O. Polverino and F. Zullo. Connections between scattered linear sets and MRD-codes. \textit{Bull. Inst. Combin. Appl.} \textbf{89} (2020), 46--74.

\bibitem{rottey} S. Rottey and G. Van de Voorde. Pseudo-ovals in even characteristic and ovoidal Laguerre planes. {\em J. Combin. Theory, Ser. A} {\bf 129} (2015), 105--121.
\bibitem{jon}
{J. Sheekey} and {G. Van De Voorde.}
Rank-metric codes, linear sets, and their duality.
\textit{Des. Codes Cryptogr.} \textbf{88} (2020), 655--675.

\bibitem{sziklai} P. Sziklai. On small blocking sets and their linearity. {\em J. Combin. Theory, Ser. A} {\bf 115} (2008), 1167--1182.

\bibitem{zini} G. Zini and F. Zullo. Scattered subspaces and related codes. arXiv:2007.04643
\end{thebibliography}
\end{document}